\documentclass[11pt]{amsart}
\usepackage{amsmath}
\usepackage{enumerate}
\usepackage{amssymb}
\usepackage{pb-diagram}
\usepackage{enumerate}
\usepackage{graphicx}
\usepackage[active]{srcltx}
\usepackage[lite]{amsrefs}

\newtheorem{theorem}{Theorem}[section]
\newtheorem{thm}[theorem]{Theorem}
\newtheorem{prop}[theorem]{Proposition}

\newtheorem{claim}[theorem]{Claim}
\newtheorem{fact}[theorem]{Fact}
\newtheorem{cor}[theorem]{Corollary}
\newtheorem{proposition}[theorem]{Proposition}
\newtheorem{lemma}[theorem]{Lemma}

\theoremstyle{definition}
\newtheorem{defn}[theorem]{Definition}

\theoremstyle{remark}
\newtheorem{rmk}[theorem]{Remark}
\newtheorem{remark}[theorem]{Remark}

\newcommand{\Z}{\mathbb{Z}}
\newcommand{\N}{\mathbb{N}}

\newcommand{\R}{\mathbb{R}}

\newcommand{\CM}{\mathcal M}
\newcommand{\sub}{\subseteq}

\newcommand{\ra}{\rangle}
\newcommand{\la}{\langle}

\newcommand{\noi}{\noindent}
\newcommand{\CH}{\mathcal H}

\newcommand{\Vdef}{$\bigvee$-definable }

\newcommand{\CV}{\mathcal V}

\newcommand{\CU}{\mathcal U}
\newcommand{\res}{\ensuremath{\upharpoonright}}

\newcommand{\Rarr}{\ensuremath{\Rightarrow}}

\newcommand{\bb}[1]{\ensuremath{\mathbb{#1}}}

\newcommand{\Lam}{\ensuremath{\Lambda}}

\newcommand{\cal}[1]{\ensuremath{\mathcal{#1}}}

\title[Definable quotients]{Definable quotients of locally definable groups}
\author{Pantelis~E.~Eleftheriou}
\address{CMAF, Universidade de Lisboa,
Av. Prof. Gama Pinto 2, 1649-003 Lisboa, Portugal} \email{pelefthe@uwaterloo.ca}
\thanks{The first author was supported by the Funda\c{c}\~ao para a Ci\^encia e a Tecnologia
grants SFRH/BPD/35000/2007 and PTDC/MAT/101740/2008}

\author{Ya'acov Peterzil}
\address{Department of Mathematics, University of Haifa, Haifa, Israel}
\email{kobi@math.haifa.ac.il}

\begin{document}

\subjclass[2010]{03C64, 03C68, 22B99}
\keywords{O-minimality, locally definable groups, definable
quotients, type-definable groups}
\date{\today}

\begin{abstract} We study locally definable abelian groups $\CU$ in various settings and examine conditions under
which the quotient of $\CU$ by a discrete subgroup might be definable. This turns
out to be related to the existence of the type-definable subgroup $\CU^{00}$ and to
the divisibility of $\CU$.
\end{abstract}
 \maketitle

\section{Introduction}

This is the first of two papers (originally written as one) around
groups definable in o-minimal expansions of ordered groups. The
ultimate goal of this project is to reduce the analysis of such
groups  to semi-linear groups and to groups definable in o-minimal
expansions of real closed fields. This reduction is carried out in
the second paper (\cite{ep-sbd}). In the current paper, we prove a
crucial lemma in that perspective, Theorem \ref{excursion} below.
This theorem is proved by analyzing \Vdef abelian groups in
various settings and investigating when such groups have definable
quotients of the same dimension. The analysis follows closely
known work on definably compact groups. We make strong use of
their minimal type-definable subgroups of bounded index, and of
the solution to so-called Pillay's conjecture in various settings.
\\

In the rest of this introduction we recall the main definitions and state the
results of this paper.\medskip

\emph{Until Section \ref{sec-div}, and unless stated otherwise, $\cal M$ denotes a sufficiently
saturated, not necessarily o-minimal, structure.}\medskip

If $\CM$ is $\kappa$-saturated, by \emph{bounded} cardinality we mean cardinality
smaller than $\kappa$. Since ``bounded'' has a different meaning in the context of
an ordered structure we use ``small'' to refer to subsets of $M^n$ of bounded
cardinality. Every small definable set is therefore finite.

\subsection{\Vdef and locally definable sets}

A \emph{\Vdef group}  is a group $\la \CU,\cdot \ra$ whose universe is a directed
union
  $\CU=\bigcup_{i\in I} X_i$ of
definable subsets of $M^n$ for some fixed $n$  (where $|I|$ is
bounded) and for every $i,j\in I$, the restriction of group
multiplication to $X_i\times X_j$ is a definable function (by
saturation, its image is contained in some $X_k$). Following
\cite{ed1}, we say that $\la \CU,\cdot \ra$ is \emph{locally
definable} if $|I|$ is countable. We are mostly interested here in
{\em definably generated} groups, namely \Vdef groups which are
generated as a group by a definable subset. These groups are
locally definable. An important example of such groups is the
universal cover of a definable group (see \cite{edel2}).  In
\cite[Section 7]{HPP} a more general notion is introduced, of an
Ind-definable group, where the $X_i$'s are not assumed to be
subsets of the same sort and there are definable maps which
connect them to each other.

A map $\phi:\CU\to \CH$ between \Vdef (locally definable) groups is called
{\em \Vdef (locally definable)}  if for every definable $X\sub \CU$ and $Y\sub \CH$,
$graph(\phi)\cap (X\times Y)$ is a definable set. Equivalently, the restriction of $\phi$  to any definable set is definable.

\begin{remark}\label{Meq} If in the above definition,
instead of $M^n$ we allow all $X_i$'s to be subsets of a fixed
sort $S$ then the analogous definition of groups and maps works in
$\CM^{eq}$. This will allow us to discuss locally definable maps
from a locally definable group $\CU$ onto an interpretable group
$\CV$.
\end{remark}
\subsection{Compatible subgroups}

\begin{defn}
(See \cite{ed1}) For a \Vdef group $\CU$, we say that $\CV\sub \CU$ is {\em a
compatible subset of $\CU$} if for every definable $X\sub \CU$, the intersection
$X\cap \CV$ is a definable set (note that in this case $\CV$ itself is a bounded
union of definable sets).
\end{defn}

Clearly, the only compatible \Vdef subsets of
a definable group are the definable ones. Note that if $\phi:\CU\to \CV$ is a \Vdef
homomorphism between \Vdef groups then $\ker(\phi)$ is a compatible \Vdef normal
subgroup of $\CU$. Compatible subgroups are used in order to obtain \Vdef quotients,
but for that we need to restrict ourselves to locally definable groups. Together with \cite[Theorem 4.2]{ed1}, we have:

\begin{fact}\label{edmundo} If $\CU$ is a locally definable group and $\CH\sub \CU$ a locally definable normal
subgroup then $\CH$ is a compatible subgroup of $\CU$ if and only if there exists a
locally definable surjective homomorphism of locally definable  groups $\phi:\CU\to
\CV$ whose kernel is $\CH$.
\end{fact}

\subsection{Connectedness}

If $\cal M$ is an o-minimal structure and $\CU\sub M^n$ is a \Vdef group then, by \cite[Theorem 4.8]{BaOt},
 it can be endowed with a manifold-like topology $\tau$, making it into a topological group. Namely, there exists a
bounded collection $ \{U_i:i\in I\}$ of definable subsets of $\CU$, whose union
equals $\CU$, such that each $U_i$ is in definable bijection with an open subset of
$M^k$ ($k=\dim\CU$), and the transition maps are continuous. The group operation and
group inverse are continuous with respect to this induced topology. Moreover, the
$U_i$'s are definable over the same parameters which define $\CU$.
 The topology $\tau$ is determined by
the ambient topology of $M^n$ in the sense that at every generic
point of $\CU$ the two topologies coincide. From
now on, whenever we refer to a topology on $G$, it is $\tau$ we are considering.

\begin{defn} (See \cite{BE}) In an o-minimal structure, a \Vdef group $\CU$ is called {\em
connected} if there exists no \Vdef compatible subset $\emptyset\subsetneqq
\CV\subsetneqq \CU$ which is both closed and open with respect to the group
topology.
\end{defn}

\subsection{Definable quotients}

\begin{defn} Given a \Vdef group $\CU$ and $\Lambda_0\sub \CU$ a
normal subgroup, we say that $\CU/\Lambda_0$ is \emph{definable}  if there exists a
definable  group $\overline{K}$ and a surjective \Vdef homomorphism $\mu:\CU\to
\overline{K}$ whose kernel is $\Lambda_0$.
\end{defn}


One can define the notion of an \emph{interpretable quotient} by
replacing `` $\overline{K}$ definable" by ``$\overline{K}$
interpretable" in the above definition. Note, however, that in
case $\CM$ is an o-minimal
 structure and $\CU$ is locally definable,  such as in Section \ref{sec-div} below,
by \cite[Corollary 8.1]{ed1}, the group $\CU$ has strong definable choice for
definable families of subsets of $\CU$. Namely,  for every definable family of
subsets of $\CU$, $\{X_t:t\in T\}$, there is a definable function $f:T\to \bigcup
X_t$ such that for every $t\in T$, $f(t)\in X_t$ and if $X_{t_1}=X_{t_2}$ then
$f(t_1)=f(t_2)$. In particular, every interpretable quotient of $\CU$ would be
definably isomorphic to a definable group.

\subsection{Results}

Our results concern the existence of the type-definable group
$\CU^{00}$, for a \Vdef abelian group $\CU$. Recall (\cite[Section
7]{HPP}) that for a definable, or \Vdef group $\CU$, we write
$\CU^{00}$ for the smallest, if such exists, type-definable
subgroup of $\CU$ of bounded index. In particular we require that
$\CU^{00}$ is contained in a definable subset of $\CU$. From now
on we use the expression ``{\em $\CU^{00}$ exists}'' to mean that
``there exists a smallest type-definable subgroup of $\CU$ of
bounded index, which we denote by $\CU^{00}$''. Note that a type
definable subgroup $\CH$ of $\CU$ has bounded index if and only if
there are no new cosets of $\CH$ in $\CU$ in elementary extensions
of $\CM$.

When $\CU$ is a definable group in a NIP structure, then $\CU^{00}$ exists (see
Shelah's theorem in \cite{Shelah}). When $\CU$ is a \Vdef group in a NIP structure
or even in an o-minimal one, then $\CU^{00}$ may not always exist. However, if we
assume that {\em some}  type-definable subgroup of bounded index exists, then there
is a smallest one (see \cite[Proposition 7.4]{HPP}). Recall that a definable $X\sub
\CU$ is called \emph{left generic} if boundedly many translates of $X$ cover $\CU$.
In Section \ref{sec-bdd}, we prove the following theorem for \Vdef groups:\\

\noindent\textbf{Theorem \ref{NIP-vdefinable}.} \emph{Let $\CU$ be an abelian \Vdef group in a NIP structure.
 If the definable non-generic sets in $\CU$ form an
ideal and $\CU$ contains at least one definable generic set, then $\CU^{00}$ exists.}\\

We also prove (Corollary \ref{defcomp}) that when we work in o-minimal expansions of
ordered groups,
 for a \Vdef abelian group which contains a definable generic set and is generated by a definably compact
 set, the non-generic definable subsets do form an ideal (this is a generalization of the  same result
  from \cite{Pet-Pi} for definably compact group, which itself relies heavily on work
 in \cite{Dolich}).

In Section \ref{sec-div}, we use these results to establish the equivalence of the following conditions.\\

\noindent\textbf{Theorem \ref{finalcor}.} \emph{Let $\CU$ be a connected
abelian \Vdef group in an o-minimal expansion of an ordered group, with $\CU$
 definably generated.  Then there is $k\in \bb N$ such that the following are equivalent:\\
\noi (i)  $\CU$ contains a definable generic set.\\
\noi (ii) $\CU^{00}$ exists.\\
\noi (iii) $\CU^{00}$ exists and $\CU/\CU^{00}\simeq \R^k\times {\mathbb T}^r$, where $\bb T$ is the circle group and $r\in \bb N$.\\
\noi (iv) There is a definable group $G$, with $\dim G=\dim \CU$,  and a \Vdef
surjective homomorphism $\phi:\CU\to G$.}

\emph{If $\CU$ is generated by a definably compact set, then (ii) is strengthened by the condition that $k+r=\dim \CU$.}\\

We conjecture, in fact, that the conditions of Theorem \ref{finalcor} are always true.\\

\noindent \textbf{Conjecture A.} \emph{ Let $\CU$ be a connected abelian \Vdef group
in an
o-minimal structure, which is definably generated. Then\\
\noindent (i) $\CU$ contains a definable generic set.\\
\noindent (ii) $\CU$ is divisible.}\\

We do not know if Conjecture A is true, even when $\CU$ is a
subgroup of a definable group. We do show that it is sufficient to
prove (i) under restricted conditions, in order to deduce the full
conjecture. In a recent paper (see \cite{ep-ND}) we prove that
Conjecture A holds for definably generated subgroups of $\la
R^n,+\ra$, in an o-minimal expansion of a real closed field $R$.

Finally, we derive the theorem that is used in \cite{ep-sbd}.\\

\noindent\textbf{Theorem \ref{excursion}.} \emph{Let $\CU$ be a connected abelian  \Vdef
group in an o-minimal expansion of an ordered group, with $\CU$ definably generated. Assume that $X\sub \CU$ is a definable set and
$\Lam\leqslant \CU$  is a finitely generated subgroup such that $X+\Lam=\CU$.}

\emph{Then there is a subgroup $\Lam'\sub \Lam$ such that $\CU/\Lam'$ is a definable group.}

\emph{If $\CU$ is generated by a definably compact set, then $\CU/\Lam'$ is moreover definably compact.}\\

\subsection{Notation}\label{notation}
Given a group $\la G, \cdot\ra$
 and a set $X\sub G$, we denote, for every $n\in \bb{N}$,
\begin{equation}
X(n)=\overbrace{XX^{-1} \cdots XX^{-1}}^{n-\text{times}}\notag
\end{equation}

We assume familiarity with the notion of definable compactness. Whenever we write that
a set is definably compact, or definably connected, we assume in particular that it is definable.

\subsection{Acknowledgements.} We wish to thank Elias Baro, Alessandro
Berarducci,  David Blanc, Mario Edmundo and Marcello Mamino for discussions which
were helpful during our work. We thank the anonymous referee for a careful reading
of the original manuscript.

\section{\Vdef groups and type-definable subgroups of bounded index}\label{sec-bdd}

 {\em In this section, unless stated otherwise,  $\CM$ denotes a sufficiently saturated, not necessarily
o-minimal, structure}.\smallskip

\subsection{Definable quotients of \Vdef groups}

We begin with a criterion for definability (and more generally
interpretability) of quotients.

\begin{lemma}\label{conditions}
Let $\la \CU,\cdot\ra$ be a \Vdef group and $\Lambda_0$ a small normal subgroup of
$\CU$. Then the following are equivalent:
\begin{enumerate}
\item The quotient $\CU/\Lam_0$ is interpretable in $\CM$.

\item  There is a definable $X\sub \CU$ such that (a) $X\cdot \Lambda_0 =\CU$ and
(b) for every definable $Y\sub \CU$, $Y \cap \Lambda_0$ is finite.

\item There is a definable $X\sub \CU$ such that (a) $X\cdot \Lambda_0 =\CU$ and (b)
$X\cap \Lambda_0$ is finite. \end{enumerate}
\end{lemma}
\begin{proof}
(1 \Rarr 2). We assume that there is a \Vdef surjective
$\mu:\CU\to \overline{K}$ with kernel $\Lam_0$, and $\overline{K}$
interpretable. By saturation, there is a definable subset $X\sub
\CU$ such that $\mu(X)=\overline{K}$ and hence $X\cdot \Lambda_0=
\CU$. Given any definable $Y\sub \CU$,  the restriction of $\mu$
to $Y$ is definable and thus
 the small set $\ker({\mu}_{\res Y})=Y \cap \Lambda_0$ is definable and, hence,
finite.

(2 \Rarr 3). This is obvious.

(3 \Rarr 1) We claim first that for every definable $Y\sub \CU$, the set  $Y\cap
\Lambda_0$ is finite. Indeed, since $Y\sub X\cdot\Lambda_0$ and $\Lambda_0$ is
small,  by saturation there exists a finite $F\sub \Lambda_0$ such that $Y\sub
X\cdot F$. We assume that $X\cap \Lambda_0$ is finite, and since $F$ is a finite
subset of $\Lam_0$ it follows that $(X\cdot F)\cap \Lam_0$ is finite which clearly
implies $Y\cap \Lam_0$ finite.

Fix a finite $F_1=XX^{-1}\cap \Lam_0$ and $F_2=XXX^{-1}\cap \Lam_0$.

 We now define on $X$ an equivalence relation $x\sim y$ if and only if
$xy^{-1}\in \Lam_0$ if and only if $xy^{-1}\in F_1$. This is a definable relation
since $F_1$ is finite. We can also define a group operation on the equivalence
classes: $[x]\cdot [y]=[z]$ if and only if $xyz^{-1}\in \Lambda_0$ if and only if
$xyz^{-1}\in F_2$. The interpretable group we get, call it $\overline{K}$,  is
clearly isomorphic to $\CU/\Lambda_0$, and we have a \Vdef homomorphism from $\CU$
onto $\overline{K}$, whose kernel is $\Lambda_0$.
\end{proof}

We will return to definable quotients of \Vdef groups in Section \ref{sec-div}.
We now focus on the existence of $\CU^{00}$ for a \Vdef group $\CU$.

\subsection{Subgroups of bounded index of \Vdef groups}\label{subgroups}

Let $\CU$ be a \Vdef group in an o-minimal structure. It is not
always true that $\CU$ has some type-definable subgroup of bounded
index. For example, consider a sufficiently saturated ordered
divisible abelian group $\la G,<, +\ra$ and in it take an infinite
increasing sequence of  elements $0<a_1<a_2<\cdots$ such that, for every
$n\in \N$, we have $na_i<a_{i+1}$. The subgroup $\bigcup_i (-a_i, a_i)$ of $G$ is a \Vdef group which does not have any
type-definable subgroup of bounded index. However, as is shown in
\cite{HPP} (see Proposition 6.1 and Proposition 7.4), if $\CU$
does have some type-definable subgroup of bounded index then it
has a smallest one; namely $\CU^{00}$ exists.

Our goal here is to show, under various assumptions on $\CU$, that
the ideal of non-generic definable sets gives rise to
type-definable subgroups of bounded index.

As is shown in \cite{Pet-Pi}, using Dolich's results in \cite{Dolich}, if $G$ is a
definably compact, abelian group in an o-minimal expansion of a real closed field
then the non-generic definable sets form an ideal. Later, it was pointed out in
\cite{ElSt} and \cite[Section 8]{pet-sbd} that the same proof works in expansions of
groups. We start by re-proving an analogue of the result for \Vdef groups (see Lemma
\ref{Dolich} below). We first define the corresponding notion of genericity and
prove some basic facts about it.

\begin{defn} Let $\CU$ be a \Vdef group. A definable $X\sub \CU$ is called {\em
left-generic} if there is a small subset $A\sub \CU$ such that $\CU=\bigcup_{g\in A}
gX$. We similarly define {\em right-generic}. The set $X$ is called {\em generic} if
it is both left-generic and right-generic.
\end{defn}

It is easy to see that a definable $X\sub \CU$ is generic if and only if for every
definable $Y\sub \CU$, there are finitely many translates of $X$ which cover $Y$.

 \begin{fact}\label{generic0}
 (1) If $\CU$ is a \Vdef group, then every \Vdef subgroup of
 bounded index is a compatible subgroup. In particular, if $X\sub \CU$ is a
 definable left-generic set, then the subgroup generated by $X$ is a compatible subgroup.

(2)  Assume that  $\CU$ is a \Vdef group in an o-minimal structure. If
$\CU$ is connected and $X\sub \CU$ is a left-generic set, then $X$ generates $\CU$.
 \end{fact}
 \proof (1) Assume that $\CV$ is a \Vdef subgroup of bounded index. We need to see that
 for every definable $Y\sub \CU$, the set $Y\cap \CV$ is definable. Since $\CV$ has
 bounded index in $\CU$ its complement in $\CU$ is also a bounded union of definable sets, hence
 a \Vdef set. But then $Y\cap \CV$ and $Y\setminus \CV$ are both \Vdef sets, so by
 compactness $Y\cap \CV$ must be definable.

(2) Assume now that $\CU$ is a \Vdef connected group in an o-minimal structure and
$X\sub \CU$ is a left-generic set. By (1), the group $\CV$ generated by $X$ is
compatible, of bounded index. But then $\dim \CV=\dim \CU$, so by \cite[Proposition 1]{BE},  $\CV=\CU$.
 \qed

\begin{fact} \label{generic1}
Let $\la \CU,+\ra $ be an abelian, definably generated group. If $X\sub \CU$ is a
definable set then $X$ is generic if and only if there exists a finitely generated
(in particular countable) group $\Gamma\leqslant \CU$ such that $\CU=X+\Gamma$.
\end{fact}
\begin{proof} Clearly, if $\Gamma $ exists then $X$ is generic. For the converse, assume that $\CU$ is  generated by the
definable set $Y\sub \CU$, with $0\in Y$. Because $X$ is generic in $\CU$, there is a
finite set $F\sub \CU$ such that the sets $-Y$, $Y$ and $X+X$ are all contained in $X+F$.

Let $Y(n)$ be as in the notation from Section \ref{notation}.
If we now let $\Gamma$ be the group generated by $F$, then $\CU=\bigcup_n
Y(n)=X+\Gamma$.
\end{proof}

We next show that under some suitable conditions we can guarantee the
 existence of $\CU^{00}$. We do it first in the general context of NIP theories. We recall  a definition \cite{Pet-Pi}:

\begin{defn} Given a \Vdef group $\CU$ and a definable set $X\sub \CU$,
$$Stab_{ng}(X)=\{g\in \CU: gX\Delta X
\mbox{ is non-generic in } \CU \}.$$ \end{defn}

\begin{theorem}\label{NIP-vdefinable} Let $\CU$ be an abelian $\bigvee$-definable group
in a NIP structure $\CM$. Assume that the non-generic definable subsets of $\CU$
form an ideal and that $\CU$ contains some definable generic set. Then for any
definable generic set $X$, the set $Stab_{ng}(X)$ is a type-definable group and has
bounded index in $\CU$. In particular, by \cite[Proposition 7.4]{HPP}, $\CU^{00}$ exists.
\end{theorem}
\proof The fact the definable non-generic sets form an ideal implies that for every
definable set $X$, the set $Stab_{ng}(X)$ is a subgroup. Note however that if $X$ is
a non-generic set then $Stab_{ng}(X) =\CU$ and therefore will not in general be
type-definable (unless $\CU$ itself was definable).

We assume now that $X\sub \CU$ is a definable generic set and show that $Stab_{ng}(X)$
is type-definable. First note that for every $g\in \CU$, if $gX\Delta X$ is
non-generic, then in particular $gX\cap X\neq \emptyset$  and therefore $g\in
XX^{-1}$. It follows that $Stab_{ng}(X)$ is contained in $XX^{-1}$.

Next, note that a subset of $\CU$ is generic if and only if finitely many translates
of it cover $X$ (since $X$ itself is generic). Now, for every $n$, we consider the
statement in $g$: ``$n$ many translates of $gX\Delta X$ do not cover $X$''. Here
again we note that for $h(gX\Delta X)\cap X$ to be non-empty we must have $h\in
XX^{-1}\cup X(gX)^{-1}$. Hence, it is sufficient to write the first-order formula
saying that for every $h_1,\ldots, h_n\in XX^{-1}\cup X(gX)^{-1}$, $X\nsubseteq
\bigcup_{i=1}^n h_i(gX\Delta X)$.  The union of all these formulas for every $n$,
together with the formula for $XX^{-1}$ is the type which defines $Stab_{ng}(X)$.

It remains to see that $Stab_{ng}(X)$ has bounded index in $\CU$. This is a similar
argument to the proof of \cite[Corollary 3.4]{HPP} but in that paper the amenability
of definable groups and, as a result, the fact that every generic set has positive
measure, played an important role. Since a generic subset of  a \Vdef group may
require infinitely many translates to cover the group, we cannot a-priori conclude
that it has positive measure, even if the group is amenable.
 Assume then towards contradiction that $Stab_{ng}(X)$ had unbounded index and fix a small elementary substructure $\CM_0$
 over which all data is definable. Then we can find a
sequence  $g_1,\ldots, g_n,\ldots\in \CU$ of indiscernibles over
$\CM_0$, which are all in different cosets of $Stab_{ng}(X)$. In
particular, it means that $g_iX \Delta g_jX$ is generic, for
$i\neq j$.

Consider now the sequence $X_i=g_{2i}X\Delta g_{2i+1}X$, $i\in \N$. By NIP, there is
a $k$, such that the sequence $\{X_i:i\in \mathbb N\}$ is $k$-inconsistent.

Consider now the type $tp(g_i/M_0)$ and find some $M_0$-definable set $W$ containing
$g_i$. Because of indiscernibility, all $g_i$'s are in $W$. It follows that all the
$g_iX$, and therefore also all $X_i$, are contained in $WX$. Because each $X_i$ is
generic, finitely many translates of $X_i$ cover $WX$. By indiscernibility,   there
is some $\ell$ such that for every $i$ there are $\ell$-many translates of $X_i$
which cover $WX$.

We then have countably many sets $X_i\sub WX$, such that on one hand the intersection
of every $k$ of them is empty and on the other hand there is some $\ell$ such that
for each $i$, $\ell$-many translates of $X_i$ cover $WX$. To obtain a contradiction
it is sufficient to prove the following lemma (it is here that we need to find an
alternative argument to  the measure theoretic one):
\begin{lemma} Let $G$ be an arbitrary abelian group, $A\sub G$ an
arbitrary subset. For every $k$ and $\ell$ there is a fixed number $N=N(k,\ell)$
such that there are at most $N$ subsets of $A$ with the property that each covers
$A$ with $\ell$-many translates and every $k$ of them have empty intersection.
\end{lemma}
\proof We are going to use the following fact about abelian groups, taken from
\cite{KT} (see problems 7 and 16 on p. 82):


\begin {fact}\label{abelian}  For every abelian group $G$, and for every set $A\sub
G$ and
 $m$, it is not possible to find $A_1,\dots A_{m+1}\sub A$ pairwise disjoint
 such that each $A_i$ covers $A$ by $m$-many translates.
\end{fact}

Returning to the proof of the lemma,  we are going to show that $N=k\ell$ works.
Assume for contradiction that there are $k\ell+1$ subsets $X_1, \dots, X_{kl+1}$ of
$A$, each covering $A$ by $\ell$-many translates, with an empty intersection of
every $k$ of them. We work with the group $G'=G\times C_{k}$, where
$C_{k}=\{0,\ldots, k-1\}$ is the cyclic group. For $i=1,\ldots, k\ell+1$, we
define $Y_i\sub G'$ as follows: For $x\in G$ and $n\in \mathbb N$, we have $(x,n)\in
Y_i$ if and only if $x\in X_i$ and $n$ is the maximum number such that for some distinct
$i_1,\ldots, i_n<i$, we have $x\in X_{i_1}\cap\cdots \cap X_{i_n}\cap X_i$. Notice
that even though $i$ might be larger than $k$, because of our assumption that
every $k$ sets among the $X_i$'s intersect trivially, the maximum $n$ we pick is
indeed at most $k-1$. Note also that the projection of each $Y_i$ on the first
coordinate is $X_i$.

We claim that the $Y_i$'s are pairwise disjoint. Indeed, if $x\in X_i\cap X_j$ and
$i<j$ then by the definition of the sets, if $(x,n)\in Y_i$ and $(x,n')\in Y_j$ then
$n<n'$, so $Y_i\cap Y_j=\emptyset$.

Now, let $A'=A\times C_{k}$. We claim that each $Y_i$ covers $A'$ by $k\ell$-many
translates. Indeed, if $A\sub \bigcup_{j=1}^\ell g_{ij}\cdot X_i$ then $$A'\sub
\bigcup_{p\in C_{k}}\bigcup_{j=1}^\ell (g_{ij},p)\cdot Y_i.$$

We therefore found $N+1$ pairwise disjoint subsets of $A'$, each covering $A'$ in
$N$ translates, contradicting Fact \ref{abelian}.\qed

Thus, as pointed out above we reached a contradiction, so
$stab_{ng}(X)$ does have bounded index in $\CU$. This ends the
proof of Theorem \ref{NIP-vdefinable}.\qed

\begin{rmk}\label{remark2} The last theorem implies that for a \Vdef abelian group $\la \CU,+\ra$
in a NIP structure, if the non-generic definable sets form an
ideal, then $\CU^{00}$ exists if and only if $\CU$ contains a
definable generic set (we have just proved the right-to-left
direction. The converse is immediate since every definable set
containing $\CU^{00}$ is generic).
\end{rmk}

We are now ready to show  (Corollary \ref{defcomp} below) that
when we work in o-minimal expansions of ordered groups,
 for a \Vdef abelian group which contains a definable generic set and is generated by a definably compact set,
 the assumptions of Theorem \ref{NIP-vdefinable} are satisfied.
We begin by proving that we can obtain Dolich's result in this
setting.

\begin{fact}\label{Dolich1} Let $\CM$ be an o-minimal expansion of an ordered group
and let $\CM_0\preccurlyeq\CM$ be a small elementary submodel.
If $\CU$ is a \Vdef group over $\CM_0$ and $X_t\sub \CU$ is a $t$-definable,
definably compact set such that $X_t\cap M_0=\emptyset$, then there are $t_1,\ldots,
t_k$, all of the same type as $t$ over $M_0$ such that $X_{t_1}\cap \cdots\cap
X_{t_k}=\emptyset$.
\end{fact}
\begin{proof} We need to translate the problem from the group topology to the $M^n$-topology.
As we already noted it is shown  in \cite{BaOt} that $\CU$ can be covered by a fixed
collection of $\CM_0$-definable open sets $\bigcup_i V_i$ such that each $V_i$ is
definably homeomorphic to an open subset of $M^n$. By logical compactness, $X_t$ is
contained in finitely many $V_i$'s, say $V_1,\cdots,V_m$. Now, by definable
compactness, we can  replace each of the $V_i$'s by an open set $W_i$ such that
$Cl(W_i)\sub V_i$ and $X_t$ is still contained in $W_1,\ldots, W_m$. Each
$X(i)=X_t\cap Cl(W_i)$ is definably compact and we finish the proof as in
\cite[Lemma 3.10]{ElSt}.
\end{proof}

For a \Vdef group $\CU$, we call a definable $X\sub \CU$ {\em relatively definably
compact} if the closure of $X$ in $\CU$ is definably compact. Clearly, $X$ is
relatively definably compact if and only if it is contained in some definably
compact subset of $\CU$.
\begin{lemma}\label{Dolich} Let $\CM$ be an o-minimal expansion of an ordered group.
 Assume that $\CU$ is
a $\bigvee$-definable abelian group, and $X,Y\sub \CU$ are definable, with $X$
relatively definably compact. If $X$ and $Y$ are non-generic, then $X\cup Y$ is still
non-generic.
\end{lemma}
\proof This is just a small variation on the work in \cite{Pet-Pi}. Because
commutativity plays only a minor role we use multiplicative notation for possible
future use.

We may assume that $\CU$ contains a definable generic set (otherwise, the conclusion is trivial).

 We need to prove that if $X\sub \CU$ is definable, relatively definably compact and non-generic,
 and if $Z\supseteq X$ is definable and generic then $Z\setminus X$ is generic.

Fix $\CM_0$ over which all sets are definable. Without loss of generality, $X$ is
definably compact (since the closure of a non-generic set is non-generic).

We first prove the result for $Z$ of the form $W\cdot W$, when $W$ is generic. Since
$X$ is not generic, no finitely many translates of $X$ cover $W$ (because $W$ is
generic). It follows from logical compactness that there is $g\in W$ such that
$g\notin \bigcup_{h\in M_0} hX$. Changing roles, there is $g\in W$ such that
$Xg^{-1}\cap M_0=\emptyset$. We now apply Fact \ref{Dolich1} to the definably
compact set $Xg^{-1}$. It follows that there are $g_1,\ldots, g_r$, all realizing
the same type as $g$ over $\CM_0$, so in particular all are in $W$, such that
$Xg_1^{-1}\cap \cdots \cap Xg_r^{-1}=\emptyset.$ This in turn implies that
$\bigcup_{i=1}^r (W\setminus Xg_i^{-1})=W$. For each $i=1,\ldots,r$ we have
$$W\setminus Xg_i^{-1}=(Wg_i\setminus X)g_i^{-1}\sub
(WW\setminus X)g_i^{-1}.$$Therefore, it follows that $W$ is
contained in the finite union $\bigcup_{i=1}^r (WW\setminus X)
g_i^{-1}$ and since $W$ is generic it follows that $W W\setminus
X$ is generic, as needed (it is here that commutativity is used,
since left generic sets and right generic sets are the same).

We now consider an arbitrary definable  generic set $Z\sub \CU$, with $X\sub Z$
non-generic. Because $Z$ is generic, finitely many translates of $Z$ cover $Z\cdot
Z$. Namely, $Z\cdot Z\sub \bigcup_{i=1}^t h_iZ$. If $X'=\bigcup _{i=1}^t h_iX$ then
$X'$ is still non-generic (and relatively definably compact), so by the case we have
just proved, $ZZ\setminus X'$ is generic. However this set difference is contained in
$$\bigcup_{i=1}^t h_iZ\setminus \bigcup_{i=1}^t h_iX\subseteq \bigcup_{i=1}^t h_i(Z\setminus X),$$
hence this right-most union is generic. It follows that $Z\setminus X$ is
generic.\qed\vskip.2cm

\begin{cor}\label{defcomp} Let $\CM$ be an o-minimal expansion of an ordered group. Assume that $\CU$ is
a $\bigvee$-definable abelian group which contains a definable generic set and is generated by a definably compact set. Then
the definable non-generic subsets of $\CU$ form an ideal.
\end{cor}
\begin{proof}
Every
definable subset of $\CU$ must be relatively definably compact, because it is
contained in some definably compact set. Then apply Lemma \ref{Dolich}.
\end{proof}

\section{Divisibility, genericity and definable quotients}\label{sec-div}
\emph{In this section, $\cal{M}$ is a sufficiently saturated o-minimal expansion of an ordered group.}\vskip.2cm

\begin{proposition} \label{exponent}
If $\CU$ is an infinite  \Vdef  group of positive dimension, then it has unbounded
exponent. In particular, for every $n$, the subgroup of $n$ torsion points,
$\CU[n]$, is small.
\end{proposition}
\proof By the Trichotomy Theorem (\cite{pest-tri}), there exists a neighborhood of
the identity which is in definable bijection with an open subset of $R^n$ for some
real closed field $R$, or of $V^n$ for some ordered vector space $V$ (we use here
the definability of a group operation near the identity of $\CU$).

In the linear case, the group operation of $\CU$ is locally isomorphic near
$e_{\CU}$ to $+$ near $0\in M^n$ (see \cite[Proposition 4.1 and Corollary 4.4]{ElSt} for a similar argument).
Clearly then the map $x\mapsto kx$ is non-constant.

Assume then that we are in the field case. Namely, we assume that some definable
neighborhood $W$ of $e$ is definably homeomorphic to an open subset of $R^n$, with
$e$ identified with $0\in \R^n$, and that a real closed field whose universe is a
subset of $W$ is definable in $\CM$. The following argument was suggested by S.
Starchenko. If $M(x,y)=xy$ is the group product of elements near $e$, then it is
$R$-differentiable and its differential at $(e,e)$ is $x+y$. It follows that the
differential of the map $x\mapsto x^n$ is $nx$. Therefore, for every $n$, the map
$x\mapsto x^n$ is not the constant map.

As for the last clause, note first that $\CU[n]$ is a compatible \Vdef subgroup of
$\CU$ because its restriction to every definable set is obviously definable (by the
formula $nx=0$). Because $\CU[n]$ has exponent at most $n$, it follows from what we
have just proved that its dimension must be zero, so its intersection with every
definable set is finite.\qed

\begin{rmk} Although we did not write down the details, we believe that the above result is actually true
without any assumptions on the ambient o-minimal $\CM$. This can be seen by
expressing a neighborhood of $e_{\CU}$ as a direct product of  neighborhoods, in
cartesian powers of orthogonal real closed fields and  ordered vector spaces.
\end{rmk}

Assume that $\CU=\bigcup_{i\in I}X_i$ and that $\CU^{00}$ exists. Given  the
projection $\pi:\CU\to \CU/\CU^{00}$, we define the \emph{logic topology} on
$\CU/\CU^{00}$ by: $F\sub\CU/ \CU^{00}$ is closed if and only if for every $i\in I$,
$\pi^{-1}(F)\cap X_i$ is type-definable. We first prove a general lemma.
\begin{lemma} \label{preimage} Let $\CU$ be a locally definable group for which $\CU^{00}$ exists and let $\pi:\CU\to \CU/\CU^{00}$ be
the projection map.
 If $K_0\sub \CU/\CU^{00}$ is a compact set, then $\pi^{-1}(K_0)$ is contained in a
definable subset of $\CU$.
\end{lemma}
\begin{proof}  We write $\CU=\bigcup_{n\in \N}X_n$, and we assume that the union is increasing.
If the result fails then there is a sequence $k_n\to \infty$ and $x_n\in
X_{k_n}\setminus X_{k_n-1}$ such that $\pi(x_{n})\in K_0$. Since $K_0$ is compact we
may assume that the sequence $\pi(x_{n})$ converges to some $a\in K_0$. The set
$\pi^{-1}(a)$ is a coset of $\CU^{00}$ and therefore contained in some definable set
$Z\sub \CU$. Since $a$ can be realized as the intersection of countably many open
sets, there is, by logical compactness, some open neighborhood $V\ni a$ in $\CU/\CU^{00}$ such
that $\pi^{-1}(V)\sub Z$. But then, the whole tail of the sequence $\{\pi(x_n)\}$
belongs to $V$ and therefore the tail of $\{x_n\}$ is contained in $Z$,
contradicting our assumption on the sequence.\end{proof}

\begin{claim}\label{reduce-compact}
Let $\CU$ be an abelian locally definable group. Then there exists a definable
torsion-free subgroup $H\sub \CU$ such that every definable subset of $\CU/H$ is relatively
definably compact. If, in addition, $\CU$ is definably generated, then
 $\CU/H$ can be generated by a definably compact set.
\end{claim}
\begin{proof}
As can easily be verified, for a definably generated \Vdef group $\CV$, the following are equivalent:
 (a) every definable subset of $\CV$ is relatively definably compact, (b) every definable path in $\CV$ has limit points in $\CV$.
  A \Vdef group with property (b) was called in \cite{ed1} ``definably compact". In Theorem 5.2 of the same reference,
   it was shown that if $\CV$ is a \Vdef group which is not definably compact, then $\CV$ contains a $1$-dimensional
   torsion-free definable subgroup $H_1$. Now, if $\CU$ is abelian, then by Fact \ref{edmundo}, $\CU/H_1$ is definably isomorphic to
   a locally definable definable group.
   Using induction on $\dim(\CU)$, we see that $\CU$ contains a definable torsion-free
subgroup
   $H$ such that $\CU/H$ is definably compact in the above sense.

If in addition, $\CU$ is definably generated then $\CU/H$ is also definably
generated by some set $X$. By replacing $X$ with $Cl(X)$ we conclude that
   $\CU/H$ is generated by a definably compact set.
\end{proof}

\begin{prop} \label{divisible} Let $\CU$ be a connected abelian $\bigvee$-definable
 group, which is definably generated. If $\CU^{00}$ exists, then
\begin{enumerate}

\item The group $\CU/\CU^{00}$, equipped with the logic topology, is isomorphic to
$\R^k\times K$, for some compact group $K$. (Later we will see that $K\simeq \mathbb
T^r$ where $\mathbb T$ is the circle group and $r\in \bb N$).

\item $\CU$ and $\CU^{00}$ are divisible.

\item $\CU^{00}$ is
torsion-free.
\end{enumerate}
\end{prop}

\begin{proof}

(1)  Let us denote the group $\CU/\CU^{00}$ by $L$.
 By \cite[Lemma 2.6]{BOPP} (applied
to $\CU$ instead of $G$ there), the image of every definable,
definably connected subset of $\CU$ under $\pi$ is a connected
subset of $L$. As in the  proof of Theorem 2.9 in \cite{BOPP}, the
group $L$ is locally connected, and since $\CU$ is connected, the
group $L$ must actually be connected.

Since $\CU$ is generated by a definable set, say $X\sub \CU$, its image $\pi(\CU)=L$
is generated by $\pi(X)$ which is a compact set ($\pi(X)$ is a quotient of $X$ by a
type-definable equivalence relation with bounded quotient, see \cite{Pillay}).
Hence, the group $L$ is so-called compactly generated. By \cite[Theorem 7.57]{HM},
the group $L$ is then isomorphic, as a topological group, to a direct product
$\R^k\times K$, for some
 compact abelian group $K$. This proves (2).

In what follows, we use $+$ for the group operation of
$\CU$ and write $\CU$ as an increasing countable union  $\bigcup_{k=1}^{\infty}
X(k)$ (with $X(k)$ as in the notation from Section \ref{notation}).


(2) Let us see that $\CU$ is divisible. Given $n\in \N$, consider the map $z\mapsto
nz:\CU\to \CU$. For a subset $Z$ of $\CU$, let $nZ$ denote the image of $Z$ under
this map.  The kernel of this map is $\CU[n]$. By Proposition \ref{exponent},
$\CU[n]$ must have dimension $0$, and therefore by connectedness
$\dim(n\CU)=\dim(\CU)$.

 Since $\CU$ is connected, by \cite[Proposition 1]{BE} it is  sufficient   to show that for every $n$, the group
$n\CU$ is a compatible subgroup of $\CU$, namely that for every definable $Y\sub
\CU$, the set $Y\cap n\CU$ is definable.

We claim  that $Y\cap n\CU$ is contained in $nX(j)$ for some $j$. Assume towards a contradiction
 that this fails.  Then for every $j$ there exists
$x_j\in \CU$ such that $nx_j\in Y\setminus nX(j)$. Hence, $x_j\notin X(j)$ and
therefore there is a sequence $k_j\to \infty$ such that $x_j\in X(k_j)\setminus
X(k_j-1)$ and $nx_j\in Y$. Consider the projection $\pi(Y)$ and $\pi(x_j)$ in $L$.
Because $Y$ is definable the set $\pi(Y)$ is compact.

By Lemma \ref{preimage}, because the sequence $\{x_j\}$ is not contained in any
definable subset of $\CU$, its image $\{\pi(x_j)\}$ is not contained in any compact
subset of $L$. At the same time, $n\pi(x_j)$ is contained in the compact set
$\pi(Y)$. However, since $L$ is isomorphic to $\R^k\times K$, for a compact group
$K$, the map $x\mapsto nx$ is a proper map on $L$ and hence this is impossible. We
therefore showed that $$Y\cap n\CU\sub nX(j)\sub n\CU,$$ and so $Y\cap n\CU=Y\cap
nX(j)$ which is a definable set. We can conclude that the group $n\CU$ is a
compatible subgroup of $\CU$, of the same dimension and therefore $n\CU=\CU$. It
follows that $\CU$ is divisible.

Let us see that $\CU^{00}$ is also divisible. Indeed, consider the map $x\mapsto nx$
from $\CU$ onto $\CU$. It sends $\CU^{00}$ onto the group $n\CU^{00}$ and therefore
$[\CU:\CU^{00}]\leq [\CU:n\CU^{00}]$. Since $\CU^{00}$ is the smallest
type-definable subgroup of bounded index we must have $n\CU^{00}=\CU^{00}$, so
$\CU^{00}$ is divisible.

(3) This is a repetition of an argument from  \cite{Pet-Pi}. Because $\CU^{00}$
exists there is a definable generic set $X\sub \CU$ which we now fix. By Theorem
\ref{NIP-vdefinable}, the group $Stab_{ng}(X)$ contains $\CU^{00}$,  so it is
sufficient to prove that for every $n$, there is a definable $Y\sub \CU$ such that
$Stab_{ng}(Y)\cap \CU[n]=\{0\}.$ We do that as follows. Because $\CU$ is divisible,
the \Vdef map $h\mapsto nh$ is surjective. By compactness, there exists a definable
$Y_1\sub \CU$ which maps onto $X$. However, since $\CU[n]$ is compatible and has
dimension zero, every element of $X$ has only finitely many pre-images in $Y_1$. By
definable choice, we can find a definable $Y\sub Y_1$ such that the map $h\mapsto
nh$ induces a bijection from $Y$ onto $X$. The set $Y$ is generic in $\CU$ as well
(since its image is generic and the kernel of the map has dimension zero) and for
every $g\in \CU[n]$ we have $(g+Y)\cap Y=\emptyset$. Hence, the only element of
$\CU[n]$ which belongs to $Stab_{ng}(Y)$ is $0$. It follows that $\CU^{00}$ is
torsion-free.
\end{proof}

As a corollary, we can formulate the following criterion for
recognizing $\CU^{00}$, generalizing results from \cite{BOPP} and
\cite{HPP}:
\begin{prop}\label{criterion-tor} Let $\CU$ be a connected abelian \Vdef group which is definably generated.
Assume that $H\leqslant \CU$ is type-definable of bounded
index. Then $H=\CU^{00}$ if and only if $H$ is torsion-free.

In particular, if $\CU$ is torsion-free then $\CU^{00}$, if it exists, is the only type-definable subgroup of bounded index.
\end{prop}
\begin{proof} Since $H$ is type-definable of bounded index, by \cite[Proposition 7.4]{HPP} $\CU^{00}$ exists.

If $H=\CU^{00}$, then by Proposition \ref{divisible} it is torsion-free.

 For the converse, assume that
$H\leqslant \CU$ is torsion-free.  We let $L=\CU/\CU^{00}$, equipped with the logic
topology. Because $\CU^{00}\leqslant H$, the map $\pi:\CU\to L$ sends the type-definable group $H$ onto a compact subgroup of $L$. If $\pi(H)$ is non-trivial
(namely, $H\neq \CU^{00}$) then $\pi(H)$ has torsion. However, $\ker(\pi)=\CU^{00}$
is divisible (see Proposition \ref{divisible}) and therefore $H$ has torsion.
Contradiction.\end{proof}

\begin{lemma}\label{equivalences} Let $\CU$ be a connected abelian $\bigvee$-definable
 group, which is definably generated.
 Then the following are equivalent.\begin{enumerate}
  \item $\CU$ contains a definable generic
 set.
\item $\CU^{00}$ exists.

 \item $\CU^{00}$ exists and $\CU/\CU^{00}\simeq \R^k\times K$,
for some $k\in \mathbb N$ and a compact group $K$.

 \item There exists a definable group $G$
 and a \Vdef surjective homomorphism $\phi:\CU\to G$ with $\ker(\phi)\simeq
 \bb{Z}^{k'}$, for some $k'\in \mathbb N$.

 \item There exists a definable group $G$
 and a \Vdef surjective homomorphism $\phi:\CU\to G$.

 \end{enumerate}
Assume now that the above hold. If $k$ is as in (3) and $\phi:\CU\to G$ and $k'$ are
as in (4), then $k=k'$.
\end{lemma}

\begin{proof}
(1) $\Rightarrow$ (2). Note first that by Claim \ref{reduce-compact}, the group
$\CU$ has a definable torsion-free subgroup $H$ with $\CU/H$ definably generated by
a definably compact set. Because $\CU$ contains a definable generic set so does
$\CU/H$. By Corollary \ref{Dolich}, the definable non-generic sets in $\CU/H$ form
an ideal, so by Theorem \ref{NIP-vdefinable}, $(\CU/H)^{00}$ exists. Its pre-image
in $\CU$ is a type definable subgroup of bounded index which is also torsion-free
(since $H$  and $(\CU/H)^{00}$ are both torsion-free). By Proposition
\ref{criterion-tor} this pre-image equals $\CU^{00}$.

(2) $\Rightarrow $ (3). By Proposition \ref{divisible}.

(3) $\Rightarrow$ (4). Let $L=\R^k\times K$ and $\pi_{\CU}:\CU\to L$ be the
 projection map (whose kernel is $\CU^{00}$).

 We now fix generators
 $z_1,\ldots, z_k\in \R^k$ for $\Z^k$, and find $u_1,\ldots, u_k\in \CU$ with
 $\pi_{\CU}(u_i)=(z_i,0)$. If we let $\Gamma\leqslant \CU$ be the subgroup generated by
 $u_1,\ldots, u_k$ then $\pi_{\CU}(\Gamma)=\Z^k$. Note that since $z_1,\ldots, z_k$
 are $\Z$-independent, the restriction of $\pi_{\CU}$ to $\Gamma$ is injective,
 namely $\Gamma\cap \CU^{00}=\{0\}$.

By Lemma \ref{preimage}, there is a definable $X\sub \CU$ such that
$\pi_{\CU}^{-1}(K)\sub X$. It follows from \cite[Lemma 1.7]{BOPP} that the set
$\pi_{\CU}(X)$ contains not only $K$ but also an open neighborhood of $K$. But then,
there is an $m$ such that $m\pi_{\CU}(X)+\Z^k=L$. This implies that
$\pi_{\CU}(mX+\Gamma)=L$ and hence $mX+\CU^{00}+\Gamma\sub mX+X+\Gamma=\CU$. We let
$Y=mX+X$ and then $Y+\Gamma=\CU$.

We claim that $Y\cap \Gamma$ is finite. Indeed, if $Y\cap \Gamma$ were infinite
then, since $\pi_{\CU}$ is injective on $\Gamma$, the set $\pi_{\CU}(Y)\cap \Z^k$ is
infinite, contradicting the compactness of  $\pi_{\CU}(Y)$. We can now apply Lemma
\ref{conditions} and conclude that there is a definable group $G$ and a \Vdef
surjective homomorphism $\phi:\CU\to G$ whose kernel is $\Gamma$.

(4) $\Rightarrow$ (5) is clear.

(5) $\Rightarrow$ (1).  By logical compactness,
there is a definable $X\sub \CU$ such that $\phi(X)=G$. But then $X+\ker(\phi)=\CU$,
and since $\ker(\phi)=\bb{Z}^{k'}$ is small, $X$ is generic in $\CU$.\vskip.2cm

Assume now that the conditions hold, $k$ is as in (3), and $\phi:\CU\to G$ and $k'$
are as in (4). We will prove that $k=k'$. Consider the map $\pi_U: U\to \R^k \times
K$ and let $\Gamma$ be the image of $\ker (\phi)$ under $\pi_U$.

We first claim that $k\le k'$. Let $X\sub \CU$ be so that $\phi(X)=G$. Then $X+\ker(\phi)=\CU$. Thus, $\pi_\CU(X)+\Gamma= \R^k \times K$. Let $Y$ and $\Gamma'$ be the projections of $\pi_\CU(X)$ and $\Gamma$, respectively, into $\bb{R}^k$. We have $Y+\Gamma'=\bb{R}^k$. The set $\pi_\CU(X)$ is compact and so $Y$ is also compact.

We let $\lambda_1,\ldots, \lambda_{k'}$ be the generators of $\ker(\phi)$ and let $v_1,\ldots,
v_{k'}\in \R^k$ be their images in $\Gamma'$. If $H\sub \R^k$ is the real subspace
generated by $v_1,\ldots, v_{k'}$ then $Y+H=\R^k$, and therefore, since $Y$ is compact,
we must have $H=\R^k$. This implies that $k\le k'$.

Now let us prove that $k'\le k$. Note first that $\ker(\phi)\cap \CU^{00}=\{0\}$. Indeed,  take any definable set
$X\sub \CU$ containing $\CU^{00}$. Then, since $\phi\res X$ is definable, we must have
$\ker(\phi)\cap \CU^{00}\sub \ker(\phi) \cap X$ finite. However, by Proposition \ref{divisible},
the group $\CU^{00}$ is torsion-free, hence $\ker(\phi)\cap \CU^{00}=\{0\}$.

It follows that $\Gamma=\pi_{\CU}(\ker \phi)$ is of rank $k'$. It is also discrete.
Indeed, using $X$ as above we can find another definable set $X'$ whose image $\pi_{\CU}(X')$
  contains an open neighborhood of $0$ and no other elements of $\Gamma$.

Now, since $K$ is compact, no element of $\Gamma$ can be in $K$ and therefore  the
projection of $\Gamma$ onto $\Gamma'\sub \R^k$ is an isomorphism. Furthermore,
$\Gamma'$ is also discrete, which implies that $k'\le k$.
\end{proof}
At the end of this section, we conjecture that the above  conditions always hold.


The result below is proved in \cite[Theorem 8.2]{B-M} for $\CU$
the universal covering of an arbitrary definably compact group $G$
in o-minimal expansions of real closed fields.
\begin{prop} \label{o-minimalcase} Let $\CU$ be a
connected abelian $\bigvee$-definable group, which is definably
generated. Let $G$ be a  definable
 group and $\phi:\CU\to G$ a surjective \Vdef
homomorphism with $\ker(\phi)\simeq \bb{Z}^k$.

Then $\CU^{00}$ exists,  $\ker(\phi)\cap \CU^{00}=\{0\}$ and
$\phi(\CU^{00})=G^{00}$. Furthermore there is a topological
covering map $\phi':\CU/\CU^{00}\to G/G^{00}$, with respect to the
logic topologies, such that the following diagram commutes.
\begin{equation}\begin{diagram}
\node{\CU}\arrow{s,l}{\pi_{\CU}}\arrow{e,t}{\phi}\node{G}\arrow{s,r}{\pi_G}\\
\node{\CU/\CU^{00}}\arrow{e,t}{\phi'}\node{G/G^{00}}
\end{diagram}\end{equation}

The group $\CU/\CU^{00}$, equipped with the logic
topology, is isomorphic to $\R^k\times\mathbb T^{r}$, for $\mathbb T$ the circle
group and $r\in \bb N$. If $\CU$ is generated by a definably compact set, then $k+r=\dim(\CU)$. If, moreover, $\CU$ is torsion-free, then $\CU/\CU^{00}\simeq \R^{\dim \CU}$.
\end{prop}
 \begin{proof} By Lemma \ref{equivalences}, $\CU^{00}$ exists. Let $\Gamma=\ker(\phi)$.
We first claim that $\Gamma\cap \CU^{00}=\{0\}$. Indeed, take any definable set
$X\sub \CU$ containing $\CU^{00}$. Then, since $\phi\res X$ is definable, we must have
$\Gamma\cap \CU^{00}\sub \Gamma \cap X$ finite. However, by Proposition \ref{divisible},
the group $\CU^{00}$ is torsion-free, hence $\Gamma\cap \CU^{00}=\{0\}$.

We claim that $\phi(\CU^{00})=G^{00}$. First note that since $\CU^{00}$ has bounded
index in $\CU$ and $\phi$ is surjective, the group $\phi(\CU^{00})$ has bounded
index in $G$. Because $\Gamma \cap \CU^{00}=\{0\}$ the restriction of $\phi$ to
$\CU^{00}$ is injective and hence $\phi(\CU^{00})$ is torsion-free. By \cite{BOPP},
we must have $\phi(\CU^{00})=G^{00}$.

By \cite{Pillay}, we have
$$G/G^{00}\simeq {\mathbb T}^l,$$
for some $l\in \bb N$.
We now consider $\pi_G:G\to G/G^{00}$ and define $\phi':\CU/\CU^{00}\to G/G^{00}$ as
follows: For $u\in \CU$, let $\phi'(\pi_{\CU}(u))=\pi_G(\phi(u))$. Since
$\phi(\CU^{00})=G^{00}$ this map is a well-defined homomorphism which makes the
above diagram commute. It is left to see that $\phi'$ is a covering map.

It follows from what we established thus far that
$\ker(\phi')=\pi_{\CU}(\Gamma)=\bb{Z}^k$. Let us see that this is a discrete
subgroup of $\CU/\CU^{00}$. Indeed,  as we already saw, for every compact
neighborhood $W\sub \CU/\CU^{00}$ of $0$, there is a definable set $Z\sub \CU$ such
that $\pi_{\CU}^{-1}(W)\sub Z$. But we already saw that $Z\cap \Gamma$ is finite and
hence $W\cap \ker(\phi')$ must be finite. It follows that $\ker(\phi')$ is discrete.

By Lemma \ref{equivalences}, $\CU/\CU^{00}$, equipped with the Logic topology, is
locally compact. Since $\phi':\CU/\CU^{00}\to G/G^{00}$ is a surjective homomorphism
with discrete kernel  it is sufficient to check that it is continuous as a map
between topological groups. If $W\sub G/G^{00}$ is open then $V=\pi_G^{-1}(W)$ is a
\Vdef subset of $G$ and hence $\phi^{-1}(V)$ is a \Vdef subset of $\CU$ (because
$\ker \phi$ is a small group). By commutation, this last set equals
$\pi_{\CU}^{-1}(\phi'^{-1}(W))$ and therefore $\phi'^{-1}(W)$ is open in
$\CU/\CU^{00}$.

By Lemma \ref{equivalences}, $\CU/\CU^{00}\simeq \R^k\times K$, for a
compact group $K$. We now have a covering map $\phi':\R^k\times K \to
G/G^{00}={\mathbb T}^{k+r}$, with
$\ker(\phi')=\Z^k\sub\R^k$. It follows that $K\simeq {\mathbb T}^r$.

If $\CU$ is generated by a definably compact set, $G$ will be definably
compact. In this case, by the work in \cite{ElSt}, \cite{HPP} and \cite{pet-sbd},
$$G/G^{00}\simeq {\mathbb T}^{\dim(G)}$$
and, hence, $k+r=\dim(G)=\dim(\CU)$.

If, moreover, $\CU$ is torsion-free, we have $r=0$.
\end{proof}

We summarize the above results in the following theorem.

\begin{thm}\label{finalcor} Let $\CU$ be a connected abelian \Vdef group which is definably generated.  Then there is $k\in \bb N$ such that the following are equivalent:

\noi (i) $\CU$ contains a  definable generic set.

\noi (ii) $\CU^{00}$ exists.

\noi (iii) $\CU^{00}$ exists and $\CU/\CU^{00}\simeq \R^k\times {\mathbb T}^r$, for some $r\in \bb N$.

\noi (iv) There is a definable group $G$, with $\dim G=\dim
\CU$,  and a \Vdef surjective homomorphism $\phi:\CU\to G$.

If in addition $\CU$ is generated by a definably compact set, then
(ii) is strengthened by the condition that $k+r=\dim \CU$.
\end{thm}
\begin{proof}
By Lemma \ref{equivalences} and Proposition \ref{o-minimalcase}.
\end{proof}

\begin{theorem}\label{excursion}
Let $\CU$ be a connected abelian \Vdef group which is definably generated. Assume that $X\sub
\CU$ is a definable set and $\Lam\leqslant \CU$  is a finitely generated subgroup such that $X+\Lam=\CU$.

Then there is a subgroup $\Lam'\sub \Lam$ such that $\CU/\Lam'$ is a definable group.

If $\CU$ generated by a definably compact set, then $\CU/\Lam'$ is moreover definably compact.
\end{theorem}

\begin{proof} Since $X+\Lam=\CU$, $X$ is generic. By Theorem \ref{finalcor}, $\CU/ \CU^{00}\simeq \R^{k}\times {\mathbb T}^r$, for some $k, r\in \bb N$. We now consider $\Delta=\pi_{\CU}(\Lambda)\sub
\R^k\times {\mathbb T}^r$ and let $\Delta'\sub \R^k$ be the projection of $\Delta$
into $\R^k$. Since $X+\Lambda=\CU$, we have $\pi_{\CU}(X)+\Delta= \R^k\times
{\mathbb T}^r$. Hence, if $Y$ is the projection of $\pi_{\CU}(X)$ into $\R^k$ then
we have $Y+\Delta'=\R^k$. The set $\pi_{\CU}(X)$ is compact and so $Y$ is also
compact.

We let $\lambda_1,\ldots, \lambda_m$ be generators of $\Lambda$ and let $v_1,\ldots,
v_m\in \R^k$ be their images in $\Delta'$. If $H\sub \R^k$ is the real subspace
generated by $v_1,\ldots, v_m$ then $Y+H=\R^k$, and therefore, since $Y$ is compact,
we must have $H=\R^k$. This implies that among $v_1,\ldots, v_m$ there are elements
$v_{i_1},\ldots, v_{i_k}$ which are $\R$-independent. It follows that
$\lambda_{i_1},\ldots, \lambda_{i_k}\in \Delta$ are $\Z$-independent. If we let
$\Lambda'$ be the group generated by $\lambda_{i_1},\ldots, \lambda_{i_k}$ then we
immediately see that the restriction of $\pi_{\CU}$ to $\Lambda'$ is injective. We
claim that $\CU/\Lambda'$ is definable.

First, let us see that for every definable $Z\sub \CU$, the set $Z\cap \Lambda'$ is
finite. Indeed, $\pi_{\CU}(Z)$ is a compact subset of $\R^k\times {\mathbb T}^r$ and
hence $\pi_{\CU}(Z)\cap (\Z v_{i_1}+\dots + \Z v_{i_k})$ is finite. Because $\pi_{\CU}|\Lambda'$
is injective it follows that $Z\cap \Lambda'$ is also finite.

We can now take a compact set $K\sub \R^k\times {\mathbb T}^r$ such that $K+\Z^k=
\R^k\times \mathbb T^r$. It follows that $\pi_{\CU}^{-1}(K)+\Lambda'=\CU$. By Lemma
\ref{preimage},
 there is a definable set $Z\sub \CU$ such that
$\pi_{\CU}^{-1}(K)\sub Z$. We now have $Z+\Lambda'=\CU$ and $Z\cap \Lambda'$ finite.
By Lemma \ref{conditions}, $\CU/\Lambda'$ is definable.

For the last clause, let $f:\CU\to \CU/\Lam'$ be the quotient map, and $X'$ a definable subset of $\CU$ such that $f(X')=\CU/\Lam'$. Since $\CU$ is generated by a definably compact set, the closure of $X'$ in $\CU$ must be a subset of a definably compact set and, hence, itself definably compact. But then it is easy to verify that $\CU/\Lam'=f(X')$ is definably compact.
\end{proof}



We end this section with a conjecture.\\

\noindent{\bf Conjecture A.} {\em   Let $\CU$ be a connected abelian \Vdef group  which is definably
generated. Then

\noindent (i) $\CU$ contains a definable generic set.

\noindent (ii) $\CU$ is divisible.}\\

Although we cannot prove the above conjecture, we can reduce it to
proving (i) under additional assumptions.\\

\noindent{\bf Conjecture B.} {\em   Let $\CU$ be a connected abelian \Vdef group,  generated by a definably compact set. Then $\CU$ contains a definable generic set.}\\

\begin{claim} Conjecture $B$ implies Conjecture $A$.
\end{claim}
\begin{proof}

We assume that Conjecture $B$ is true.

Let $\CU$ be a connected abelian \Vdef group which is definably generated. Let $\CV$ be the universal cover of $\CU$ (see
\cite{edel2}). Because $\CU$ is the homomorphic image of $\CV$
under a \Vdef homomorphism whose kernel is a set of dimension $0$,
it is sufficient to prove that $\CV$ contains a generic set and
that $\CV$ is divisible.

 The group $\CV$ is connected, torsion-free and generated by a definable set
$X\sub \CV$. We work by induction on $\dim(\CV)$.

Let $Y$ be the closure of $X$ with respect to the group topology
of $\CV$.\vskip.2cm

\noi{\bf Case 1 } The set $Y$ is definably compact.\vskip.2cm

  Since  $\CV$ is generated by $Y$, then by our standing
assumption we may conclude that $\CV$ contains a definable generic set. By Theorem \ref{finalcor} and Proposition \ref{divisible}, $\CV$ is divisible.\vskip.2cm

 \noi{\bf Case 2 } The set $Y$ is not definably compact.\vskip.2cm

 In this case, we can apply \cite[Theorem 5.2]{ed1} and obtain a definable 1-dimensional, definably
connected, divisible, torsion-free subgroup of $\CV$, call it $H$. Clearly, $H$ is a
compatible subgroup of $\CV$, hence the group $\CV/H$ is  $\bigvee$-definable, connected (\cite[Corollary 4.8]{ed1}),
torsion-free and definably generated  (by the
image of $X$ under the projection map). We have $\dim(\CV/H)<\dim \CV$, so by
induction, the conjecture holds for $\CV/H$, hence it is divisible and contains a
definable generic set $Z$. Because $H$ is divisible as well, it follows that $\CV$
is divisible. It is easy to see that the pre-image of $Z$ in $\CV$ is a definable
generic subset of $\CV$.\end{proof}

Finally, although we know that $\CU$ needs to be definably generated in order to guarantee (i) (by  Fact \ref{generic0}(2)), we do not know if the same is true for (ii).\\

\noindent{\bf Conjecture C.} {\em   Let $\CU$ be a connected abelian \Vdef group. Then $\CU$ is divisible.}


\begin{bibdiv}
\begin{biblist}
\normalsize

\bib{BE}{article}{
   author={Baro, El{\'{\i}}as},
   author={Edmundo, M{\'a}rio J.},
   title={Corrigendum to: ``Locally definable groups in o-minimal
   structures''  by
   Edmundo},
   journal={J. Algebra},
   volume={320},
   date={2008},
   number={7},
   pages={3079--3080},
}

\bib{BaOt}{article}{
   author={Baro, El{\'{\i}}as},
   author={Otero, Margarita},
   title={Locally definable homotopy},
   journal={Annals of Pure and Applied Logic},
   volume={161},
   date={2010},
   number={4},
   pages={488--503},
}


\bib{B-M}{article}{
   author={Berarducci, Alessandro},
   author={Mamino, Marcello},
   title={On the homotopy type of definable groups in an o-minimal
   structure},
   journal={J. Lond. Math. Soc. (2)},
   volume={83},
   date={2011},
   number={3},
   pages={563--586},
}

\bib{BOPP}{article}{
   author={Berarducci, Alessandro},
   author={Otero, Margarita},
   author={Peterzil, Yaa'cov},
   author={Pillay, Anand},
   title={A descending chain condition for groups definable in o-minimal
   structures},
   journal={Ann. Pure Appl. Logic},
   volume={134},
   date={2005},
   number={2-3},
   pages={303--313},
}



\bib{Dolich}{article}{
   author={Dolich, Alfred},
   title={Forking and independence in o-minimal theories},
   journal={Journal of Symbolic Logic},
   volume={69},
   date={2004},
   pages={215--240},
}
\bib{ed1}{article}{
   author={Edmundo, M{\'a}rio J.},
   title={Locally definable groups in o-minimal structures},
   journal={J. Algebra},
   volume={301},
   date={2006},
   number={1},
   pages={194--223},
}




\bib{edel2}{article}{
   author={Edmundo, M{\'a}rio J.},
   author={Eleftheriou, Pantelis E.},
   title={The universal covering homomorphism in o-minimal expansions of groups},
   journal={Math. Log. Quart.},
   volume={53},
   date={2007},
   pages={571--582},
  }




\bib{ep-sbd}{article}{
   author={Eleftheriou, Pantelis E.},
   author={Peterzil, Ya'acov},
   title={Definable groups
as homomorphic images of semilinear and field-definable groups},
   journal={preprint},
   volume={},
   date={},
   number={},
   pages={},
}


\bib{ep-ND}{article}{
   author={Eleftheriou, Pantelis E.},
   author={Peterzil, Ya'acov},
   title={Lattices in locally definable subgroups of $\la R^n,+\ra$},
   journal={preprint},
   volume={},
   date={},
   number={},
   pages={},
}

\bib{ElSt}{article}{
   author={Eleftheriou, Pantelis E.},
   author={Starchenko, Sergei},
   title={Groups definable in ordered vector spaces over ordered division
   rings},
   journal={J. Symbolic Logic},
   volume={72},
   date={2007},
   number={4},
   pages={1108--1140},
}



\bib{HM}{book}{
   author={Hofmann, Karl H.},
   author={Morris, Sidney A.},
   title={The structure of compact groups},
   series={de Gruyter Studies in Mathematics},
   volume={25},
   note={A primer for the student---a handbook for the expert},
   publisher={Walter de Gruyter \& Co.},
   place={Berlin},
   date={1998},
   pages={xviii+835},
}


\bib{HPP}{article}{
   author={Hrushovski, Ehud},
   author={Peterzil, Ya'acov},
   author={Pillay, Anand},
   title={Groups, measures, and the NIP},
   journal={J. Amer. Math. Soc.},
   volume={21},
   date={2008},
   number={2},
   pages={563--596},
}




\bib{KT}{book}{
   author={Komj{\'a}th, P{\'e}ter},
   author={Totik, Vilmos},
   title={Problems and theorems in classical set theory},
   series={Problem Books in Mathematics},
   publisher={Springer},
   place={New York},
   date={2006},
   pages={xii+514},
}




\bib{pet-sbd}{article}{
   author={Peterzil, Ya'acov},
   title={Returning to semi-bounded sets},
   journal={J. Symbolic Logic},
   volume={74},
   date={2009},
   number={2},
   pages={597--617},
}


\bib{pest-tri}{article}{
   author={Peterzil, Ya'acov},
   author={Starchenko, Sergei},
   title={A trichotomy theorem for o-minimal structures},
   journal={Proceedings of London Math. Soc.},
   volume={77},
   date={1998},
   number={3},
   pages={481--523},
}

        \bib{Pet-Pi}{article}{
   author={Peterzil, Ya'acov},
   author={Pillay, Anand},
   title={Generic sets in definably compact groups},
   journal={Fund. Math.},
   volume={193},
   date={2007},
   number={2},
   pages={153--170},
}








\bib{Pillay}{article}{
author={Pillay, A.},
     title={Type-definability, compact Lie groups, and $o$-minimality},
   journal={J. Math. Logic},
    volume={},
      date={2004},
    number={4},
     pages={147\ndash 162},
}


\bib{Shelah}{article}{
author={Shelah, S.},
     title={Minimal bounded index subgroups and dependent theories},
   journal={Proc. Amer. Math. Soc.},
    volume={136},
      date={2008},
     pages={1087\ndash 1091},
}




\end{biblist}
\end{bibdiv}
\end{document}